\documentclass[12pt, reqno, a4paper]{amsart}

\usepackage{ amssymb, amsmath, enumerate, amsfonts, amsthm, mathrsfs, url, bm, mathtools}

\usepackage{xcolor}  	
\usepackage{hyperref}
\hypersetup{
colorlinks,
   linkcolor={cyan!80!black},
   citecolor={cyan!80!black},
 urlcolor={cyan!80!black}
}

\usepackage{color}

\usepackage[margin=1in]{geometry}

\RequirePackage{doi}

\usepackage{amscd}
\usepackage{amsfonts}
\usepackage{float}
\usepackage{enumitem}
\usepackage{color}
\usepackage[
backend=biber,
style=alphabetic,
]{biblatex}
\usepackage{bookmark}

\renewbibmacro{in:}{}
\DeclareFieldFormat{title}{#1}

\DeclareFieldFormat[article]{title}{\mkbibemph{#1}}       % Articles
\DeclareFieldFormat[incollection]{title}{\mkbibemph{#1}}  % Incollection chapters
\DeclareFieldFormat[book]{title}{\mkbibemph{#1}}          % Standalone books
\DeclareFieldFormat[incollection]{booktitle}{#1}          % Plain for incollection book titles
\DeclareFieldFormat[article]{journaltitle}{#1}            % Plain for journals

\AtEveryBibitem{%
  \ifentrytype{misc}{\DeclareFieldFormat{title}{\mkbibemph{#1}}}{}}

\DeclareFieldFormat{eprint:eprint}{arXiv:\href{https://arxiv.org/abs/#1}{#1}}

\DeclareFieldFormat[inproceedings]{title}{\mkbibemph{#1}}

\DeclareFieldFormat[inproceedings]{booktitle}{#1}

\usepackage{amssymb}
\newtheorem{theorem}{Theorem}[section]
\newtheorem{lemma}{Lemma}[section]

\newtheorem{corollary}{Corollary}[section]
\newtheorem{proposition}{Proposition}[section]

\theoremstyle{definition}
\newtheorem{definition}{Definition}[section]

\theoremstyle{remark}

\numberwithin{equation}{section}

\newcommand{\Mod}[1]{\ (\mathrm{mod}\ #1)}

\newcommand{\N}{\mathbb{N}}

\newcommand{\Z}{\mathbb{Z}}

\renewcommand{\leq}{\leqslant}
\renewcommand{\geq}{\geqslant}

\renewcommand{\epsilon}{\varepsilon}
\newcommand{\eps}{\varepsilon}

\begin{document}

\title[Multiplicative recurrence of Möbius transformations]{Multiplicative recurrence of Möbius transformations}

%    optional.  Re
%move any unused author tags.

%    author one information

\author{Sun-Kai Leung}
\address{D\'epartement de math\'ematiques et de statistique\\
Universit\'e de Montr\'eal\\
CP 6128 succ. Centre-Ville\\
Montr\'eal, QC H3C 3J7\\
Canada}
\curraddr{}
\email{sun.kai.leung@umontreal.ca}
\thanks{}

\author{Christian T\'afula}
\address{D\'epartement de math\'ematiques et de statistique\\
Universit\'e de Montr\'eal\\
CP 6128 succ. Centre-Ville\\
Montr\'eal, QC H3C 3J7\\
Canada}
\curraddr{}
\email{christian.tafula.santos@umontreal.ca}
\thanks{}

%    author two information
  %  \subjclass is required.

\date{}

\dedicatory{}


%    Abstract is required.
\begin{abstract}
We establish a complete characterization of multiplicative recurrence for images of the positive integers under Möbius transformations, answering a question of Donoso–Le–Moreira–Sun in the negative. As a consequence, we strengthen and extend a Diophantine approximation result of Charamaras–Mountakis–Tsinas, confirming their conjectures. %and provides strong evidence toward a conjecture of Charamaras–Mountakis–Tsinas.
\end{abstract}

\maketitle

\section{Introduction}

Recurrence is a central theme in ergodic Ramsey theory. Resolving a conjecture of Lovász, 
Sárközy \cite{MR466059} and Furstenberg \cite{MR498471} independently proved that the set of squares $\{ n^2 \,:\, n \in \mathbb{N}\}$ is a \textit{set of (additive) recurrence}.\footnote{A set $R \subseteq \mathbb{N}$ is called a \textit{set of (additive) recurrence} if for any subset $A \subseteq \mathbb{N}$ with positive upper asymptotic density 
\begin{align*}
\overline{d}(A):=\limsup_{N \to \infty} \frac{| A \cap \{1,\ldots,N \} |}{N} ,
\end{align*}
the intersection $(A-A) \cap R$ is non-empty.} Sárközy \cite{MR487031} also proved that the set of shifted primes \( \mathbb{P} - 1 \) is recurrent. %Soon after, Kamae and Mendès-France \cite{MR516154} provided a sufficient condition for a set to be recurrent. 
These combinatorial problems are not only interesting in their own
right, but they also carry ergodic significance, thanks to \textit{Furstenberg's correspondence principle} (see \cite[Theorem 2.1]{MR4594405}).
%(see Theorem \ref{thm:furstenberg}). 

It is natural to consider multiplicative analogues of these questions, a relatively young area first investigated by Bergelson \cite{MR2191223} and further developed by Frantzikinakis and Host \cite{MR3556289} more recently.

\begin{definition}[Set of multiplicative recurrence]
A set \( R \subseteq \mathbb{Q}_{>0} \setminus \{1\} \) is called a \textit{set of (measurable) multiplicative recurrence} if, for any measure-preserving action \( T = (T_n)_{n \in \mathbb{N}} \) of the semigroup \( (\mathbb{N}, \times) \) on a probability space \( (X, \mathcal{B}, \mu) \),\footnote{That is, \( T_{mn} = T_m \circ T_n \) for all \( m, n \in \mathbb{N} \).} and for any Borel set \( B \in \mathcal{B} \) with positive measure \( \mu(B) > 0 \), there exist \( m, n \in \mathbb{N} \) such that \( m/n \in R \) and
\[
\mu(T_m^{-1}B \cap T_n^{-1}B) > 0.\footnote{See \cite[Proposition 2.7]{MR4594405} for equivalent definitions in combinatorial terms.}
\]
\end{definition}

We also recall a strictly weaker notion of multiplicative recurrence.

\begin{definition}[Set of topological multiplicative recurrence]
A set $R \subseteq \mathbb{Q}_{>0} \setminus \{1\}$ is called a \textit{set of topological multiplicative recurrence} if for any finite partition 
\begin{align*}
\mathbb{N}=C_1 \sqcup \cdots \sqcup C_s,
\end{align*}
the intersection $(C_i/C_i) \cap R $ is non-empty for some integer $1\leq i\leq s.$\footnote{See \cite[Proposition 2.13]{MR4594405} for equivalent definitions in ergodic-theoretic terms.} 
\end{definition}

Using suitable equivalent definitions, one can show by Furstenberg's correspondence principle that any set of multiplicative recurrence is also a set of topological multiplicative recurrence (see also \cite[p. 730]{MR4594405}).

%Roughly speaking, we call $R \subseteq\mathbb{Q}_{>0} \setminus \{1\}$ a \textit{set of (measure) multiplicative recurrence} if for any ``multiplicatively large" subset $A \subseteq \mathbb{N},$ the intersection $(A/A) \cap R$ is non-empty (see \cite[Proposition 2.7]{MR4594405} for the precise definition). 
In a recent breakthrough toward the conjecture of Erdős and Graham on the partition regularity of Pythagorean triples (see \cite{MR2337049}), Frantzikinakis, Klurman, and Moreira \cite{frantzikinakis2023partitionregularitypythagoreanpairs} proved that for any $\ell, \ell' \in \mathbb{N},$ the sets 
\begin{align*}
\left\{  \frac{\ell(m^2-n^2)}{\ell'mn} \,:\, m,n \in \mathbb{N}, m>n\right\}\setminus \{1\}\quad \text{and} \quad
\left\{  \frac{\ell(m^2+n^2)}{\ell'mn} \,: m,n \in \mathbb{N} \right\} \setminus \{1\}
\end{align*}
are multiplicatively recurrent (see also \cite{frantzikinakis2024partitionregularitygeneralizedpythagorean}, \cite{frantzikinakis2024partitionregularityhomogeneousquadratics}, and \cite{frantzikinakis2025decompositionresultsmultiplicativeactions}). %In this paper, we focus primarily on the weaker notion, namely, \textit{topological multiplicative recurrence}.

%Ramsey theory is a branch of combinatorics which studies the inevitable emergence of specific patterns within arbitrary configurations. 

%\begin{definition}[Set of topological multiplicative recurrence]
%A set $R \subseteq \mathbb{Q}_{>0} \setminus \{1\}$ is called a \textit{set of topological multiplicative recurrence} if for any finite partition 
%\begin{align*}
%\mathbb{N}=C_1 \sqcup \cdots \sqcup C_s,
%\end{align*}
%the intersection $(C_i/C_i) \cap R $ is non-empty for some integer $1\leq i\leq s.$\footnote{See \cite{MR4594405} for equivalent definitions in ergodic-theoretic terms.} 
%\end{definition}

In this paper, we are interested in the multiplicative recurrence of images of the positive integers under Möbius transformations. Given \( a, c \in \mathbb{N} \) and \( b, d \in \mathbb{Z} \), define
\[
R_{a,b,c,d} := \left\{ \frac{an + b}{cn + d} \,:\, n \in \mathbb{N} \right\} \cap (\mathbb{Q}_{>0} \setminus \{1\}).\]
Donoso, Le, Moreira, and Sun \cite[Theorem 1.6]{MR4594405} proved that \( R_{a,b,c,d} \) is a set of topological multiplicative recurrence if $a=c$ and either \( a \mid b \) or \( a \mid d \). They further asked whether this condition is not only sufficient but also necessary (see \cite[Question 7.2]{MR4594405}). Our main result gives a complete characterization of the (measurable) multiplicative recurrence of the sets 
\( R_{a,b,c,d} \), thereby answering their question in the negative.

\begin{theorem} \label{thm:main}
Let \( a, c \in \mathbb{N} \) and \( b, d \in \mathbb{Z} \). 
If \( R_{a,b,c,d}\) is non-empty, then the following are equivalent:
\begin{enumerate}[label=\textnormal{(\roman*)}]
\item $R_{a,b,c,d}$ is a set of multiplicative recurrence.
\item $R_{a,b,c,d}$ is a set of topological multiplicative recurrence.
\item \( a = c \) and \( a \mid \mathrm{lcm}(b, d) \).
\end{enumerate}
\end{theorem}

For instance, the theorem confirms that 
\[
R_{6,3,6,2} := \left\{ \frac{6n + 3}{6n + 2} \,:\, n \in \mathbb{N} \right\} 
\]
is a set of multiplicative recurrence, which is one of the counterexamples to their question
 (see \cite[p.~759]{MR4594405} and \cite[p.~1]{charamaras2024multiplicativerecurrencelinearpatterns}). 

Recently, %under the coprimality condition \( (a, b, c, d) = 1 \), 
Charamaras, Mountakis, and Tsinas \cite[Theorem 1.1]{charamaras2024multiplicativerecurrencelinearpatterns} generalized 
the Diophantine approximation result of Klurman and Mangerel \cite[Theorem~1.1]{MR3856825} that
\begin{align*}
\liminf_{n \to \infty} |f(n+1)-f(n)|=0
\end{align*}
for any completely multiplicative function $f: \mathbb{N} \to \mathbb{S}^1$ (see \cite[Corollary 1.7]{MR4594405} for the previous generalization). As a consequence of the proof of Theorem \ref{thm:main}, we strengthen and extend their generalized result as follows.

%\begin{theorem}
%\label{thm:diophantine}
%Let \( a, c  \in \mathbb{N} \) and \( b, d \in \mathbb{Z} \). For any $\epsilon>0$, given $r\geq 1$ completely multiplicative functions \( f_1, \ldots, f_r : \mathbb{N} \to \mathbb{S}^1 \), define the set.
%\begin{align*}
%\mathcal{A}(f_1,\ldots,f_r;\epsilon):=\{ n \in \mathbb{N} \,:\, |f_j(an+b)-f_j(cn+d)| < \epsilon \quad \forall \, 1 \leq j \leq r\}
%\end{align*}
%Then, the following are equivalent:
%\begin{itemize}
%    \item $\mathcal{A}(f_1,\ldots,f_r;\epsilon)$ has positive lower asymptotic density for every  \( f_1, \ldots, f_r : \mathbb{N} \to \mathbb{S}^1 \);\medskip

 %   \item \( a = c \) and \( a \mid \mathrm{lcm}(b, d) \).
%\end{itemize}
%\end{theorem}

\begin{theorem}
\label{thm:diophantine}
Let \( r, a, c  \in \mathbb{N} \) and \( b, d \in \mathbb{Z} \). Then the set
\begin{align*}
\mathcal{A}_{a,b,c,d}(f_1,\ldots,f_r;\epsilon):=\{ n \in \mathbb{N} \,:\, |f_j(an+b)-f_j(cn+d)| < \epsilon \quad \forall \, 1 \leq j \leq r\}
\end{align*}
has positive lower asymptotic density\footnote{It is effectively computable.} for any $\epsilon>0$ and any completely multiplicative functions \( f_1, \ldots, f_r : \mathbb{N} \to \mathbb{S}^1 \)
if and only if \( a = c \) and \( a \mid \mathrm{lcm}(b, d) \).
\end{theorem}
\noindent
\textit{Remark.} Adapting the proof of the sufficiency part (see Section~\ref{section:thm2suff}), one can prove the following generalization: if \( a = c \) and \( a \mid \mathrm{lcm}(b, d) \), then for any measure-preserving action \( T = (T_n)_{n \in \mathbb{N}} \) of the semigroup \( (\mathbb{N}, \times) \) on a probability space \( (X, \mathcal{B}, \mu) \), and for any Borel set \( B \in \mathcal{B} \) with \( \mu(B) > 0 \), the set of integers \( n \geq 1 \) for which
\[
\mu(T_{an+b}^{-1}B \cap T_{cn+d}^{-1}B) > 0
\]
has positive lower asymptotic density.

For $r=1$, we improve \cite[Theorem 1.1]{charamaras2024multiplicativerecurrencelinearpatterns} by establishing positive lower asymptotic density in place of the weaker upper logarithmic density. For pretentious or finitely generated functions, this also improves \cite[Proposition 5.1]{charamaras2024multiplicativerecurrencelinearpatterns} by upgrading the density from lower logarithmic to lower asymptotic. 
For general $r \in \mathbb{N},$ the following result is an immediate consequence of Theorem \ref{thm:diophantine}.

\begin{corollary}
Let \( r, a  \in \mathbb{N} \) and \( b, d \in \mathbb{Z} \). If \( a \mid \mathrm{lcm}(b, d) \), then for any completely multiplicative functions $f_1, \ldots, f_r:\mathbb{N} \to \mathbb{S}^1,$ we have
\begin{align*}
\liminf_{n \to \infty} \max_{1 \leq j \leq r} |f_j(an+b)-f_j(an+d)| = 0.
\end{align*}
\end{corollary}

 This confirms the simultaneous Diophantine approximation conjecture posed by Charamaras, Mountakis, and Tsinas \cite[Conjecture 2]{charamaras2024multiplicativerecurrencelinearpatterns}, which is currently out of reach via pretentious methods. See \cite[Section 1.4]{charamaras2024multiplicativerecurrencelinearpatterns} for a brief discussion of the challenges.

%As noted by the authors, the coprimality condition 
%$(a,b,c,d)=1$ in all theorems of \cite{charamaras2024multiplicativerecurrencelinearpatterns} is inessential. 
Under the coprimality condition $\operatorname{gcd}(a,b,d)=1,$\footnote{If \( \gcd(a, b, d) = g  \), %(unless \( b \) or \( d = 0 \)),
then
\[
\frac{an + b}{an + d} = \frac{(a/g)n + (b/g)}{(a/g)n + (d/g)},
\]
so one may assume without loss of generality that \( \gcd(a, b, d) = 1 \).
In this case, \( a \mid \operatorname{lcm}(b, d) \) if and only if \( a \mid bd \).
} they further conjecture that the assumption $a \mid bd$ is also sufficient for \( R_{a,b,a,d} \) to be a set of multiplicative recurrence (see \cite[Conjecture 1]{charamaras2024multiplicativerecurrencelinearpatterns}), which is now confirmed by  %In particular, $R_{a,b,a,d}$ are sets of topological multiplicative recurrence. 
Theorem~\ref{thm:main}.
%We start by showing that if $a\neq c$, then $R_{a,b,c,d} \subseteq S_{\alpha,\beta} := (\alpha,\beta)\cup (\beta^{-1},\alpha^{-1})$ for some $0<\alpha<\beta<1$, so since $\chi(S_{\alpha,\beta}) < \infty$ (by Donoso--Le--Moreira--Sun \cite[Lemma 3.1]{MR4594405}), $R_{a,b,c,d}$ is not of topological multiplicative recurrence (Corollary \ref{cor:ac}). For $a=c$, we first show that if $a\nmid \mathrm{lcm}(b,d)$ then $\chi(R_{a,b,a,d}) < \infty$ (Proposition \ref{prop:vpa}). For $a=c$ and $a\mid \mathrm{lcm}(b,d)$, we will prove the stronger statement that $\omega(R_{a,b,a,d}) = \infty$, where $\omega(R)$ denotes the clique number of $G(R)$.

In this paper, we introduce a novel graph-theoretic perspective on multiplicative recurrence that is both elementary and surprisingly powerful. Given a non-empty set \( R \subseteq \mathbb{Q}_{>0} \setminus \{1\} \), one can associate a graph \( G = G(R) = (V, E) \) with vertex set \( V := \mathbb{N} \) and edge set
\[
E := \left\{ \{m,n\} \subseteq \mathbb{N}^2 \,:\, \frac{m}{n} \in R \right\}.
\]

Let \( \chi(G) \) denote the \emph{chromatic number} of \( G \), that is, the minimum number of colors needed to color the vertices with no adjacent vertices sharing a color. Then \( R \) is, by definition, a set of topological multiplicative recurrence if and only if \( \chi(G(R)) = \infty \). 

Moreover, let $\omega(G)$ denote the \textit{clique number} of $G,$ that is, the number of vertices in a maximum clique in $G.$ Then \( R \) is, by a pigeonhole argument, a set of multiplicative recurrence if  \( \omega(G(R)) = \infty \) (see Lemma \ref{lem:moreira}). Note that the chromatic number $\chi(G) $ is always at least the clique number $\omega(G).$

In what follows, we write \( G_{a,b,c,d} := G(R_{a,b,c,d}) \). 
Then the implication \textnormal{(ii)} $\implies$ \textnormal{(iii)} in Theorem~\ref{thm:main} is equivalent to the statement that \( \chi(G_{a,b,c,d}) = \infty \) only if \( a = c \) and \( a \mid \mathrm{lcm}(b, d) \). Also, the implication \textnormal{(iii)} $\implies$ \textnormal{(ii)} is weaker than the statement that \( \omega(G_{a,b,c,d}) = \infty \) if \( a = c \) and \( a \mid \mathrm{lcm}(b, d) \). Finally, the implication \textnormal{(i)} $\implies$ \textnormal{(ii)} is trivial.

\medskip
\noindent\textit{Notation.} 
 Given a graph $G=(V,E)$, we denote by $\chi(G), \omega(G)$ the \textit{chromatic number} and the \textit{clique number} of $G$, respectively. Given a subset $W \subseteq V,$ we denote by $G[W]$ the \textit{induced subgraph} on $W$, with vertex set $W$ and edge set $\{ \{ w, w'\} \in E \,:\, w,w' \in W \}$. Given a prime \( p \), we denote by \( v_p : \mathbb{Q} \to \mathbb{Z} \cup \{\infty\} \) the \textit{\( p \)-adic valuation}. We denote by $\lfloor \cdot  \rfloor, \lceil \cdot  \rceil, \{ \cdot\}, \| \cdot \|$ the \textit{floor function}, the \textit{ceiling function}, the \textit{fractional part function}, and the \textit{distance to the nearest integer function}, respectively. Given a set of integers $A \subseteq \mathbb{N},$ we denote by $\underline{d}(A)$ its \textit{lower asymptotic density}, defined as
\begin{align*}
\underline{d}(A):=\liminf_{N \to \infty} \frac{|A \cap \{ 1, \ldots, N\}|}{ N}.
\end{align*}

\section{Proof of Theorem \ref{thm:main}: $\textnormal{(ii)} \implies \textnormal{(iii)}$}
  In this section, we establish the implication $\textnormal{(ii)} \implies \textnormal{(iii)}$ in Theorem \ref{thm:main}; that is, if either $a \neq c$, or $a=c$ but $a \nmid \mathrm{lcm}(b,d),$ then $\chi(G_{a,b,c,d})<\infty.$ Note that if $a\neq c$, then neither $0$,$1$, nor $\infty$ are limit points of $R_{a,b,c,d}$, meaning that $R_{a,b,c,d} \subseteq (\beta^{-1},\alpha^{-1})\cup(\alpha,\beta)$ for some $1<\alpha<\beta.$ Otherwise, if $a=c$, then $a \nmid \mathrm{lcm}(b,d)$ if and only if $v_p(a)>\max\{v_p(b),v_p(d)\}$ for some prime $p.$ Therefore, it suffices to establish the following explicit bounds on the chromatic numbers.
%which is divided into Corollary \ref{cor:ac} and Proposition \ref{prop:vpa}. We begin with the %case $a \neq c.$

%In general, if neither $0$, $1$, nor $\infty$ are accumulation points of $R \subseteq $, then there exists $0 < \alpha <\beta<1$ such that
% \[ R\subseteq S_{\alpha, \beta} := (\alpha,\beta) \cup (\beta^{-1},\alpha^{-1}), \]
% and thus $\chi(R)\leq \chi(S_{\alpha,\beta})$. This is the case of $R_{a,b,c,d}$ if $a\neq c$.
% We obtain an explicit bound to $\chi(S_{\alpha,\beta})$, sharpening \cite[Lemma 3.1]{MR4594405}.
 
\begin{theorem} \label{thm:nec}
Let \( a, c \in \mathbb{N} \) and \( b, d \in \mathbb{Z} \).  
\begin{enumerate}[label=\textnormal{(\roman*)}]
    \item Suppose \( a \ne c \). Then
    \[
    \chi(G_{a,b,c,d}) \leq \left\lceil \frac{\log \alpha\beta}{\log \alpha} \right\rceil
    \]
    for any $1<\alpha<\beta$ satisfying $R_{a,b,c,d} \subseteq (\beta^{-1},\alpha^{-1}) \cup (\alpha,\beta).$

    \item Suppose \( a = c \). If there exists a prime \( p \) such that \( v_p(a) > \max\{ v_p(b), v_p(d) \} \) and \( v_p(b) = v_p(d) \), then
    \[
    \chi(G_{a,b,c,d}) \leq p^{v_p(b - d)}(p - 1).
    \]

    \item Suppose \( a = c \). If there exists a prime \( p \) such that \( v_p(a) > \max\{ v_p(b), v_p(d) \} \) and \( v_p(b) \ne v_p(d) \), then
    \[
    \chi(G_{a,b,c,d}) = 2.
    \]
\end{enumerate}
\end{theorem}

When $a \neq c,$ Donoso, Le, Moreira and Sun \cite[Lemma 3.1]{MR4594405} proved, using ergodic methods, that $R_{a,b,c,d}$ is not a set of topological multiplicative recurrence, which is equivalent to $\chi(G_{a,b,c,d})<\infty$. We refine their result by providing an explicit bound through a purely combinatorial argument. 

When $a=c,$ they also proved 
that $\chi(G_{a,b,c,d}) \leq p^{v_p(b - d)}(p - 1)$ if $p | a$ and $p \nmid bd$ for some prime $p$ (see \cite[Proposition 3.3]{MR4594405}). We generalize their result to the broader case where $ v_p(a) > \max\{ v_p(b), v_p(d) \} $ for some prime $p,$ that is, when $a \nmid \mathrm{lcm}(b,d).$

Before proving the theorem, we first bound the chromatic numbers of three classes of graphs by constructing explicit colorings, motivated by the unimodular completely multiplicative functions appearing in Section~\ref{sec:dio_nec}.

\begin{proposition} \label{prop:1}
Given \( 1 < \alpha < \beta \), define
\[
S_{\alpha,\beta} := \left\{ r \in \mathbb{Q}_{>0} \,:\, r \in (\beta^{-1}, \alpha^{-1}) \cup (\alpha, \beta) \right\},
\]
and denote \( G_{\alpha,\beta} := G(S_{\alpha,\beta}) \). Then
\[
\chi(G_{\alpha,\beta}) \leq \left\lceil \frac{\log \alpha\beta}{\log \alpha} \right\rceil .
\]
\end{proposition}

\begin{proof}
Let $t = \frac{1}{\log\alpha\beta}$. For each integer $k \geq \frac{\log \alpha\beta}{\log \alpha},$ consider the $k$-coloring $c:\N\to \{1,\ldots,k\}$ given by
  \[ c(n) = j \quad\text{if}\quad \{t\log n\} \in \bigg[\frac{j-1}{k},\, \frac{j}{k} \bigg), \]
which can be interpreted as a ``level-set coloring" of the Archimedean character $n \mapsto n^{it}.$
We claim that $c$ is a proper $k$-coloring of $G_{\alpha,\beta},$ that is, no adjacent vertices share the same color. Suppose to the contrary that there exist integers $m \neq n \geq 1$ with $c(m)=c(n)$ such that 
$m/n \in S_{\alpha,\beta}.$ Without loss of generality, we assume that $m/n \in (\alpha,\beta)$ by the symmetry of $S_{\alpha,\beta}.$ 
 %we have $m/n \notin S_{\alpha,\beta}$, which implies that $c$ is a proper $k$-coloring of $G(S_{\alpha,\beta})$. 
 %Since $m/n\in S_{\alpha,\beta}$ if and only if $n/m\in S_{\alpha,\beta}$, we assume that $m>n$ without loss of generality.
%  where $\lVert x\rVert := \min_{n\in\Z} |x-n|$. 
On one hand, since $m/n \in (\alpha,\beta)$, we have 
  \[ \left\lVert t\log \frac{m}{n} \right\rVert \geq \min\left\{\frac{\log \alpha}{\log \alpha\beta}, 1-\frac{\log \beta}{\log \alpha\beta} \right\} = \frac{\log \alpha}{\log \alpha\beta} \geq \frac{1}{k}\] 
by our choice of $t.$ 
On the other hand, since $c(m)=c(n),$ we have 
  \[ \left\lVert t\log \frac{m}{n} \right\rVert = \left\lVert t\log m - t\log n \right\rVert < \frac{1}{k}. \]
Therefore, contradiction arises and $c$ is a proper $k$-coloring of $G_{\alpha,\beta}$ as claimed. Taking 
$k=\lceil (\log \alpha\beta)/(\log \alpha) \rceil,$ the proposition follows.
\end{proof}

\begin{proposition} \label{prop:2}
Given a prime $p$ and an integer $k \geq 0,$ define
\begin{align*}
T_{p,k}:=\{ r \in \mathbb{Q}_{>0} \,:\, v_p(r)=0, v_p(r-1) \leq k \},
\end{align*}
and denote \( G_{p,k} := G(T_{p,k}) \).\footnote{ Let \( | \cdot |_p \) denote the \( p \)-adic absolute value on \( \mathbb{Q} \), defined by \( |r|_p := p^{-v_p(r)} \) for \( r \in \mathbb{Q} \). Then  
\[
T_{p,k} = \left\{ r \in \mathbb{Q}_{>0} \,:\, |r|_p = 1,\ |r - 1|_p \geq p^{-k} \right\},
\]
which can be regarded as a non-Archimedean analogue of the sets \( S_{\alpha, \beta} \).
} Then
\[
\chi(G_{p,k}) \leq p^k(p-1).
\]
\end{proposition}

\begin{proof}
Consider the $(p^{k}(p-1))$-coloring $c:\N\to (\mathbb{Z}/p^{k+1}\mathbb{Z})^{\times}$ given recursively by
  \[ c(n) = \begin{cases}
      n \Mod{p^{k+1}} &\text{if  $\,p \nmid n$},\\
  \hfil    c(n/p) &\text{if  $\,p \mid n$},
  \end{cases} \]
which appears in the proof of \cite[Proposition 3.3]{MR4594405}. One may interpret this as a ``level-set coloring" of a Dirichlet character of modulus $p^{k+1}.$ We claim that $c$ is a proper $k$-coloring of $G_{p,k},$ that is, no adjacent vertices share the same color. Suppose to the contrary that there exist integers $m \neq n \geq 1$ with $c(m)=c(n)$ such that 
$m/n \in T_{p,k}.$ On one hand, since $m/n \in T_{p,k}$ we have
\begin{align}
v_p(m)-v_p(n)=v_p(m/n)=0 \quad \text{and}  \quad v_p((m/n)-1) \leq k. \label{eq:vp} 
\end{align}
On the other hand, since $c(m)=c(n),$ if $m=p^{v_p(m)}m'$ and $n=p^{v_p(n)}n'$ for some integers $m',n' \geq 1,$ then $m' \equiv n' \Mod{p^{k+1}}$ and in particular, we have $v_p(m'-n') \geq k+1.$ Then, it follows from the leftmost equality in (\ref{eq:vp}) that
\begin{align*}
v_p \left( (m/n)-1\right) &= v_p((m'/n')-1) \\
&= v_p((m'-n')/n') \\
&= v_p(m'-n') \\
&\geq  k+1, 
\end{align*}
contradicting the rightmost inequality in (\ref{eq:vp}). Therefore, $c$ is a proper $(p^{k}(p-1))$-coloring of $G_{p,k}$ as claimed, and the proposition follows.  
\end{proof}

\begin{proposition} \label{prop:3}
Given a prime $p$ and an integer $k \geq 1,$ define
\begin{align*}
T'_{p,k}:=\{ r \in \mathbb{Q}_{>0} \,:\, |v_p(r)|=k \},
\end{align*}
and denote \( G'_{p,k} := G(T'_{p,k}) \). Then
\[
\chi(G'_{p,k}) = 2.
\]
\end{proposition}

\begin{proof}
Consider the $2$-coloring $c:\N\to \mathbb{Z}/2\mathbb{Z}$ given by 
  \[ c(n) := \begin{cases}
      0 &\text{if } \lfloor v_{p}(n)/k \rfloor \equiv 0 \Mod{2},\\
      1 &\text{if } \lfloor v_{p}(n)/k \rfloor \equiv 1 \Mod{2}.
  \end{cases} \]
  We claim that $c$ is a proper $k$-coloring of $G'_{p,k}$, that is, no adjacent vertices share the same color. Suppose to the contrary that there exist integers $m \neq n \geq 1$ with $c(m)=c(n)$ such that 
$m/n \in T'_{p,k}.$ On one hand, since $m/n \in T'_{p,k}$ we have 
$|v_p(m)-v_p(n)|=|v_p(m/n)|=k,$ so that
\begin{align*}
\left\lfloor \frac{v_p(m)}{k} \right\rfloor - \left\lfloor \frac{v_p(m)}{k} \right\rfloor  \equiv  1 \Mod{2}.
\end{align*}
On the other hand, since $c(m)=c(n),$ we have 
\begin{align*}
\left\lfloor \frac{v_p(m)}{k} \right\rfloor \equiv  \left\lfloor \frac{v_p(m)}{k} \right\rfloor  \Mod{2},
\end{align*}
a contradiction. Therefore,  $c$ is a proper $2$-coloring of $G'_{p,k}$ as claimed. Also, since $G'_{p,k}$ is not totally disconnected, the proposition follows.
%In case $\lfloor v_{p}(m)/k \rfloor = \lfloor v_{p}(n)/k\rfloor$, we would have $|v_p(m)-v_p(n)| < k$, a contradiction. But then, if $m\geq n$, we would have $|\lfloor v_{p}(m)/k \rfloor - \lfloor v_{p}(n)/k\rfloor| \geq 2$, so that $|v_p(m)-v_p(n)|>k$, contradiction. Therefore, $c$ is a proper $2$-coloring of $G'_{p,k}$ as claimed. %Since $G'_{p,k}$ is not totally disconnected, we must have $\chi(G'_{p,k})=2$ and the proposition follows.
\end{proof}

Finally, Theorem~\ref{thm:nec} follows readily from the above propositions.

\begin{proof}[Proof of Theorem \ref{thm:nec}]
Suppose \( a \ne c \). Then by definition $R_{a,b,c,d} \subseteq S_{\alpha,\beta},$ so that $G_{a,b,c,d} \subseteq G_{\alpha,\beta}.$ In particular, it follows from Proposition \ref{prop:1} that
\begin{align*}
\chi(G_{a,b,c,d}) \leq \chi(G_{\alpha,\beta}) \leq \left\lceil \frac{\log \alpha\beta}{\log \alpha} \right\rceil.
\end{align*}

Suppose now that \( a = c \). Let $p$ be a prime for which $v_p(a)>\max\{ v_p(b), v_p(d)\}.$ Write \( a = p^\alpha A \), \( b = p^\beta B \), and \( d = p^\delta D \), where \( p \nmid ABD \); that is, \( \alpha = v_p(a) \), \( \beta = v_p(b) \) and \( \delta = v_p(d) \).

If \( \beta = \delta \), then 
%\begin{align*}
%\frac{an+b}{an+d}=\frac{p^{v_p(a)-v_p(b)}(an/p^{v_p(a)})+b/p^{v_p(b)}}{p^{v_p(a)-v_p(b)}(an/p^{v_p(a)})+d/p^{v_p(b)}}
%\end{align*}
%and
%\begin{align*}
%\frac{an+b}{an+d}=1+\frac{b-d}{an+d}, 
%\end{align*}
\begin{align*}
r := \frac{an + b}{an + d} 
= \frac{p^{\alpha - \beta} A n + B}{p^{\alpha - \beta} A n + D},
\end{align*}
so \( v_p(r) = 0 \), and
\begin{align*}
r - 1 = \frac{B - D}{p^{\alpha - \beta} A n + D},
\end{align*}
so \( v_p(r - 1) = v_p(B - D) \leq v_p(b - d) \).
 Therefore, we have
 $R_{a,b,c,d} \subseteq T_{p,k}$ and  $G_{a,b,c,d} \subseteq G_{p,k}$ for $k=v_p(b-d).$ In particular, it follows from Proposition \ref{prop:2} that
\begin{align*}
\chi(G_{a,b,c,d}) \leq \chi(G_{p,k}) \leq p^{v_p(b - d)}(p - 1).
\end{align*}

%If there exists a prime \( p \) such that \( v_p(a) > \max\{ v_p(b), v_p(d) \} \) but \( v_p(b) \ne v_p(d) \), then
%\begin{align*}
%\frac{an+b}{an+d}=
%\frac{p^{v_p(a)-\min\{v_p(b), v_p(d)\}}(an/p^{v_p(a)})+b/p^{\min\{v_p(b), v_p(d)\}}}{p^{v_p(a)-\min\{v_p(b), v_p(d)\}}(an/p^{v_p(a)})+d/p^{\min\{v_p(b), v_p(d)\}}},
%\end{align*}

If \( \beta > \delta \), then 
\begin{align*}
r:=\frac{an+b}{an+d}=
\frac{p^{\alpha-\delta} An + p^{\beta-\delta}B}{p^{\alpha-\delta} An + D},
\end{align*}
so $v_p(r) = \beta-\delta$; and if $\beta<\delta$ the analogous argument gives $v_p(r) =\delta- \beta$. Either way, we obtain $v_p(r) = |v_p(b)-v_p(d)|.$  Therefore, we have $R_{a,b,c,d} \subseteq T'_{p,k}$ and $G_{a,b,c,d} \subseteq G'_{p,k}$ for $k=|v_p(b)-v_p(d)|.$  In particular, it follows from Proposition \ref{prop:3} that
\begin{align*}
\chi(G_{a,b,c,d}) \leq \chi(G'_{p,k}) = 2.
\end{align*}
Also, since \( G_{a,b,c,d} \) is not totally disconnected, we must have \( \chi(G_{a,b,c,d}) > 1 \). The proof is therefore complete.
\end{proof}

%%%%%%%%%%%%%%%%%%%%%%%%%%%%%%%%%%%%%%%%%%%%%%%%%%%%%
\section{Proof of Theorem \ref{thm:main}: \textnormal{(iii)} $\implies$ \textnormal{(i)}}

In this section, we establish the implication $\textnormal{(iii)} \implies \textnormal{(i)}$ in Theorem \ref{thm:main}. We show in fact the stronger statement that $G_{a,b,a,d}$ contains arbitrarily large cliques. 

\begin{lemma} \label{lem:moreira}
Let $R \subseteq \mathbb{Q}_{>0} \setminus \{1\}$ be non-empty. If $\omega(G(R))=\infty,$ then $R$ is a set of multiplicative recurrence.\footnote{The converse, however, is not true. For example, it is well known that $\{n^k: n\in\N\}$ is a set of (additive) recurrence for each integer $k \geq 3$, hence $R_k:= \{2^{n^k} :n\in \N\}$ is a set of multiplicative recurrence. However, a triangle in $G(R_k)$ corresponds to a non-trivial solution to $a^k+b^k = c^k$, so $\omega(G(R_k))=2$.} 
\end{lemma}

\begin{proof}
Let \( T = (T_n)_{n \in \mathbb{N}} \) be a measure-preserving action of the semigroup \( (\mathbb{N}, \times) \) on a probability space \( (X, \mathcal{B}, \mu) \), and \( B \in \mathcal{B} \) be a Borel set with positive measure \( \mu(B) > 0 \). Also, let $k=\lfloor \mu(B)^{-1} \rfloor+1.$ Then by assumption, there exists a $k$-clique $\{ n_1, \ldots, n_k\}$ in the graph $G(R)$ for some pairwise distinct integers $n_1, \ldots, n_k \geq 1;$ that is either $n_i/n_j \in R$ or $n_j/n_i \in R$ for all $1 \leq i < j \leq k.$ 

Applying Bonferroni's inequality, we have
\begin{align*}
\sum_{1 \leq i <j \leq k} \mu(T^{-1}_{n_i} B \cap T^{-1}_{n_j} B)
\geq \sum_{1 \leq i \leq k} \mu(T^{-1}_{n_i} B) -\mu \bigg( \bigcup_{1 \leq i \leq k} T^{-1}_{n_i} B \bigg).
\end{align*}
Since $T_{n_i}$ is measure-preserving for each $1 \leq i \leq k,$ the right-hand side equals
\begin{align*}
k\mu(B) -\mu \bigg( \bigcup_{1 \leq i \leq k} T^{-1}_{n_i} B \bigg)
\geq k\mu(B)-1,
\end{align*}
which is positive by our choice of $k.$ Therefore, there exist $1 \leq i< j \leq k$ such that
\begin{align*}
\mu(T^{-1}_{n_i} B \cap T^{-1}_{n_j} B)>0.
\end{align*}
Also, since either $n_i/n_j \in R$ or $n_j/n_i \in R,$ the lemma follows.
\end{proof}

Therefore, it suffices to prove the following theorem.

\begin{theorem} \label{thm:suff}
Let $a \in \mathbb{N}$ and $b \neq d \in \mathbb{Z}.$\footnote{When $b=d,$ the set $R_{a,b,a,b}$ is, by definition, empty.}
Suppose $a \, | \, \mathrm{lcm}(b,d).$ Then $\omega(G_{a,b,a,d})=\infty.$
\end{theorem}

When $a \mid b$ or $a \mid d,$ Donoso, Le, Moreira and Sun \cite[Proposition 3.6]{MR4594405} proved,  using ergodic methods, that $R_{a,b,c,d}$ is a set of topological multiplicative recurrence, which is equivalent to $\chi(G_{a,b,a,d})=\infty$. Through a purely combinatorial argument, we not only generalize their result to the broader case where \( a \mid \mathrm{lcm}(b,d) \), but also refine it by providing an algorithm for constructing arbitrarily large cliques (see the proof of Proposition \ref{prop:a=b(b-1)}). In other words, given any integer \( k \geq 2 \), one can explicitly find integers \( n_1, \ldots, n_k \geq 1 \) such that \( n_i / n_j \in R_{a,b,a,d} \cup R_{a,d,a,b} \) for all \( 1 \leq i < j \leq k \).

We reduce the general case to the special case $a = b(b-1)$ with $b\geq 2$. 

%The remainder of the section is devoted to the proof of Theorem \ref{thm:suff}. 

 \begin{lemma} \label{lem:reduction}
 Let $a \in \mathbb{N}$ and $b, d \in \mathbb{Z}.$  Suppose $b>d, \, \gcd(a,b-d)=1$ and $a \mid bd.$ Then, there exist $A,B\in\N$ with $B\geq 2$ and $A=B(B-1)$ such that $R_{a,b,a,d} \supseteq R_{A, B, A, B-1 }$. In particular, $\omega(G_{a,b,a,d}) \geq \omega(G_{A, B, A, B-1 }).$
 
 Moreover, there exist $C,D\in\N$ such that for any integer $m \geq 1,$ if $n = Cm+D$, then
 \[ \frac{an + b}{an+d} = \frac{Am+B}{Am+(B-1)}. \]
 \end{lemma}

 \begin{proof}
%If $a=1,$ then $\gcd (a,b-d)=1.$ Otherwise, suppose $a \geq 2.$ Then since $a \mid bd$ and $\gcd (a,b,d)=1,$ there exists a prime $p \mid a$ such that $p \mid bd,$ and either $p \mid b$ but $p \nmid d,$ or $p \mid d$ but $p \nmid b,$ so that $\gcd (a,b-d)=1$ as well. 
Since by assumption $\gcd (a,b-d)=1,$ there exist integers $j,k$ such that $bj \leq 0$ and $aj+(b-d)k=1.$ Therefore, for any integer $n \geq 1,$ if  $\frac{an+bk}{an+(bk-1)} >0,$ then we have
 \begin{align}
 \frac{an+bk}{an+(bk-1)} &= \frac{a(b-d)n+(b-d)bk}{a(b-d)n+(b-d)(bk-1)} \nonumber \\
 &=  \frac{a((b-d)n-bj)+b}{a((b-d)n-bj)+d}
 \in R_{a,b,a,d}. \label{eq:lin1}
 \end{align}
 %So $\frac{an+b}{an+d} = \frac{am+bk}{am+(bk-1)}$ for $m = (b-d)n-bj$.
 Since $aj+(b-d)k=1$, we have $a\mid (b-d)k-1$, which, multiplying by $bk$, implies that $a\mid bk(bk-1)-bdk^2$. Since $a\mid bd$, it follows that $a\mid bk(bk-1)$, so $R_{a,b,a,d} \supseteq R_{a, bk, a, bk-1 }$.

 Taking $n = m + T$ for some large fixed integer $T\geq 0$, we obtain
 \begin{equation}
    \frac{an+bk}{an+(bk-1)} = \frac{am + (aT+bk)}{am+(aT+bk-1)}. \label{eq:lin2}
 \end{equation}
 Thus, writing $B:=aT+bk$, we have that $R_{a,b,a,d} \supseteq R_{a,B,a,B-1}$. Choose $T$ large enough so that $B\geq 2$. Since $a\mid bk(bk-1)$, we have $a\mid B(B-1)$, which means there exists an integer $J\geq 1$ such that $Ja = B(B-1)$. Taking $m = Jr$ for integers $r \geq 1$ and writing $A := Ja$, we get $A=B(B-1)$ and
 \begin{equation}
     \frac{am+B}{am+(B-1)} = \frac{Ar+B}{Ar+(B-1)}, \label{eq:lin3}
 \end{equation} 
 so $R_{a,B,a,B-1} \supseteq R_{A,B,A,B-1}.$ Furthermore, putting \eqref{eq:lin1}, \eqref{eq:lin2}, and \eqref{eq:lin3} together, it follows that for $n = (b-d)(Jm+T) -bj$, we have
 \[ \frac{an + b}{an+d} = \frac{Am+B}{Am+(B-1)}. \]
Therefore, the lemma follows form taking $C= (b-d)J$ and $D=(b-d)T - bj$.
 \end{proof}

For later application, we prove a stronger version of Theorem \ref{thm:suff} in this special case.

\begin{proposition} \label{prop:a=b(b-1)}
Let $a, b \in \mathbb{N}$ with $a=b(b-1).$ Then
for any integer $k \geq 2,$ there exists an integer $H_k \geq 1$ depending on $a,b$ such that, for each integer $n\geq 0$, the induced subgraph $G_{a,b,a,b-1}[\{n+1, \ldots, n+H_k\}]$ contains a $k$-clique. In particular, we have $\omega(G_{a,b,a,b-1})=\infty.$
\end{proposition}

\begin{proof}
We proceed by induction on $k$. When \( k = 2 \), one can simply take \( H_2 = a + b \), since every interval of length \( H_2 \) contains a pair of consecutive integers of the form \( an + (b - 1) \), \( an + b \).
 Suppose that the statement holds for some integer $k \geq 2.$ Let 
\begin{align*}
M_k:=\prod_{1 \leq h \leq H_k}(bh+1) \quad \text{and} \quad r_k:=\frac{1}{b}(M_k-1).
\end{align*}
Then by the induction hypothesis, there exist integers $1 \leq h_1, \ldots, h_k \leq H_k$ such that the set of vertices $\{(b-1)r_k+h_1, \ldots, (b-1)r_k+h_k\}$ forms a $k$-clique in $G_{a,b,a,b-1}.$ In other words, for any integers $1 \leq i < j \leq k,$ we have
\begin{align*}
\frac{(b-1)r_k+h_i}{(b-1)r_k+h_j}=\frac{am_{ij}+b}{am_{ij}+(b-1)}
\end{align*}
for some integers $m_{ij} \geq 1.$ Since 
\begin{align*}
\frac{am_{ij}+b}{am_{ij}+(b-1)}
\end{align*}
is a reduced fraction, we have
\begin{gather*}
(b-1)r_k+h_i=g_{ij} (am_{ij}+b) \quad \text{and} \quad
(b-1)r_k+h_j= g_{ij}(am_{ij}+(b-1)),
\end{gather*}
where $g_{ij}:=\gcd((b-1)r_k+h_i, (b-1)r_k+h_j)$. It follows that $g_{ij}= h_i-h_j \leq H_{k}$ by considering their difference.
%\[ am+b = \frac{(b-1)r_k+h_i}{g_{ij}} = \frac{(b-1)r_k+h_j}{g_{ij}} + \frac{h_i-h_j}{g_{ij}} = am+(b-1) + \frac{h_i-h_j}{g_{ij}}. \] 
Let
\begin{align*}
n_{\ell} := b (M_k+H_k)! {\ell}+ r_k %\label{eq:n_t}
\end{align*}
for integers ${\ell} \geq 1.$ Then by the assumption $a=b(b-1)$, we have
\begin{align}
\frac{an_{\ell}+bh_i}{an_{\ell}+bh_j} &= \frac{an_{\ell}+b(g_{ij} (am_{ij}+b)-(b-1)r_k)}{an_{\ell}+b(g_{ij}(am_{ij}+(b-1))-(b-1)r_k)} \nonumber \\
&=  \frac{a  b (M_k+H_k)! {\ell}+bg_{ij}(am_{ij}+b)}{a  b (M_k+H_k)! {\ell}+bg_{ij} (am_{ij}+(b-1))}.  \label{eq:crazyfrac}
\end{align}
Since $g_{ij} \leq H_k < M_k+H_k,$ we have $g_{ij} \mid (M_k+H_k)!,$ so that the fraction in (\ref{eq:crazyfrac}) can be reduced to
\begin{equation*}
\frac{an_{\ell}+bh_i}{an_{\ell}+bh_j} = \frac{a \left( \frac{(M_k+H_k)!}{g_{ij}}{\ell}+m_{ij} \right)+b}{a \left( \frac{(M_k+H_k)!}{g_{ij}}{\ell}+m_{ij} \right)+(b-1)} \in R_{a,b,a,b-1}.
\end{equation*}
Therefore, the set of vertices $\{ an_{\ell}+bh_1, \ldots, an_{\ell}+bh_k \}$ forms a new $k$-clique in $G_{a,b,a,b-1}$ for each integer ${\ell} \geq 1$. 

Let ${\ell} \geq 1$ be a fixed integer. We claim that $an_{\ell}+(b-1)$ is adjacent to $an_{\ell}+bh_j$ for each integer $1 \leq j \leq k,$ so that they form a $(k+1)$-clique in $G_{a,b,a,b-1}.$ By our choice of $M_k$, $r_k$, we have $br_k+1\equiv 0 \pmod{bh+1}$ for every $0\leq h\leq H_k$, or equivalently $r_k\equiv h\pmod{bh+1}$. Hence, for each $1\leq h\leq H_k$ there exists an integer $m_h \geq 1$ such that 
\begin{align*}
n_{\ell} = b(M_k+H_k)! {\ell} +r_k = (bh+1)m_h +h.
\end{align*}
Since $a+b = b(b-1)+b = b^2$, this implies that
\begin{align}
\frac{an_{\ell}+b(h+1)}{an_{\ell}+(b-1)} &= \frac{a(bh+1)m_h+ (ah+bh+b)}{a(bh+1)m_h+(ah+b-1)} \nonumber \\
&= \frac{am_h+b}{am_h+(b-1)} \in R_{a,b,a,b-1}. \label{eq:red2}
\end{align}
Therefore, the induced subgraph $G_{a,b,a,b-1}[\{an_{\ell}+(b-1), \ldots, an_{\ell}+b(H_k+1) \}]$ contains a $(k+1)$-clique for each integer ${\ell} \geq 1$. Since $r_k+b(H_k+1) \leq ab(M_k+H_k)!, $ we can take $H_{k+1}=2ab(M_k+H_k)!,$ so that
 the induced subgraph $G_{a,b,a,b-1}[\{n, n+1, \ldots, n+H_{k+1}\}]$ contains a $(k+1)$-clique for all integers $n \geq 1$ as desired, and the proposition follows.
\end{proof}

Finally, Theorem \ref{thm:suff} follows readily from the above lemma and proposition.

\begin{proof}[Proof of Theorem \ref{thm:suff}]
 Without loss of generality, we assume that $b>d.$ If $a=1,$ then $\gcd (a,b-d)=1.$ Otherwise, suppose $a \geq 2$ and $bd \neq 0.$
 Let $g:=\gcd(a,b,d)$ and $a=ga', b=gb',d=gd'.$ Then
 \begin{align*}
 \frac{an+b}{an+d}=\frac{a'n+b'}{a'n+d'}
 \end{align*}
 for any integer $n \geq 1,$ so $G_{a,b,a,d}=G_{a',b',a',d'}.$
 Since $a' \mid \operatorname{lcm}(b',d')$ and $\gcd (a',b',d')=1,$ there exists a prime $p \mid a'$ such that $p \mid b'd',$ and either $p \mid b'$ but $p \nmid d',$ or $p \mid d'$ but $p \nmid b',$ so that $\gcd (a',b'-d')=1$ as well. Finally, suppose $a \geq 2$ and $bd = 0.$ Without loss of generality, assume that $d=0.$ Let $h:=\gcd(a,b)$ and $a=ha', b=hb'.$ Then similarly, we have $G_{a,b,a,0}=G_{a',b',a',0}$ and $\gcd (a',b')=1.$ Therefore, it suffices to prove Theorem \ref{thm:suff} for those $a \in \mathbb{N}$ and $b,d \in \mathbb{Z}$ satisfying $b>d, \gcd(a,b-d)=1$ and $a \mid \operatorname{lcm}(b,d).$ Applying Lemma \ref{lem:reduction}, the theorem follows from Proposition \ref{prop:a=b(b-1)}. 
\end{proof}

\section{Proof of Theorem \ref{thm:diophantine}: Necessity} \label{sec:dio_nec}

In this section, we establish the ``only if'' direction of Theorem~\ref{thm:diophantine}. In fact, a stronger statement holds: if either \( a \neq c \), or \( a = c \) but \( a \nmid \mathrm{lcm}(b,d) \), then there exists a completely multiplicative function \( f : \mathbb{N} \to \mathbb{S}^1 \) such that
\[
\liminf_{n \to \infty} \left| f(an + b) - f(cn + d) \right| > 0.
\]
This result was established in \cite[Propositions 2 and 3]{charamaras2024multiplicativerecurrencelinearpatterns}. For completeness, we include a streamlined proof.

Note that $a \nmid \operatorname{lcm}(b,d)$ if and only if $v_p(a)>\max\{v_p(b), v_p(d) \}$ for some prime $p.$ Therefore, it suffices to prove the following lemmas.

 \begin{lemma}
 Let $a,c\in\N$ and $b,d\in\Z$. If $a \neq c,$ then there exists a completely multiplicative function $f:\N\to \mathbb{S}^1$ such that
  \[ \liminf_{n\to\infty} |f(an+b)-f(cn+d)| > 0. \]
 \end{lemma}

 \begin{proof}
Consider the Archimedean character $f(n)=n^{it}$ for some $t\neq 2\pi k/\log(a/c)$, where $k \in \mathbb{Z}.$ We have
\begin{align*}
 \lim_{n\to\infty}  |(an+b)^{it} - (cn+d)^{it}| &= \lim_{n\to\infty} \bigg| \bigg(\frac{an+b}{cn+d}\bigg)^{it} -1\bigg| \\
 &=  \lim_{n\to\infty}  \bigg|  \exp \left(it \log \left(\frac{an+b}{cn+d} \right)\right) -1\bigg|\\
&= | \exp(it \log(a/c)) - 1| > 0,
\end{align*} 
and the lemma follows.
 \end{proof}

\begin{lemma} \label{lem:dio_a=c}
Let $a\in\N$ and $b,d\in\Z$. Suppose $v_{p}(a) > \max\{v_p(b),v_p(d)\}$ and $v_p(b) = v_p(d)$ for some prime $p.$ Then there exists a completely multiplicative function $f:\N\to \mathbb{S}^1$ such that
  \[ \liminf_{n\to\infty} |f(an+b)-f(an+d)| > 0. \]
\end{lemma}

\begin{proof}
Since
  \[ |f(an+b)-f(an+d)| = |f((a/p^{v_p(b)})n+(b/p^{v_p(b)}))-f((a/p^{v_p(b)})n+(d/p^{v_p(b)}))|, \]
we assume without loss of generality that $v_p(b)=v_p(d)=0$. Let $k=v_p(b-d).$ Then since $b\not\equiv d\pmod{p^{k+1}}$, there exists a Dirichlet character $\chi$ of modulus $p^{k+1}$ such that $\chi(b)\neq \chi(d)$ by orthogonality. Consider the modified Dirichlet character $\widetilde{\chi}: \mathbb{N} \to \mathbb{S}^1$ defined on primes by
    \[ \widetilde{\chi}(q) = \begin{cases}
        \chi(q) &\text{if }q\neq p, \\
      \hfil    1 &\text{if } q=p.
    \end{cases} \]
    For any integer $n \geq 1$, we have $b(an+d) - d(an+b) = a(b-d)n$. Since $p\mid a$ and $p^{k}\mid b-d$, we have $b(an+d) \equiv d(an+b) \Mod{p^{k+1}}$. Using the assumption $p\nmid b d$, we obtain
    \[ \widetilde{\chi}(b(an+d)) = \chi(b(an+d)) = \chi(d(an+b)) = \widetilde{\chi}(d(an+b)), \]
Therefore, we conclude that
    \begin{align*}
      |\widetilde{\chi}(an+b)-\widetilde{\chi}(an+d)| &= \bigg| \frac{\widetilde{\chi}(an+b)}{\widetilde{\chi}(an+d)} -1 \bigg| \\
      &= \bigg| \frac{\chi(b)}{\chi(d)} -1 \bigg| >0,
    \end{align*}
 and the lemma follows.   
\end{proof}

 %The following is Proposition 3.2 and 3.3 of [THE GREEKS]. The choice of functions is based on their proof. \textcolor{red}{Give them more credit: they had the idea of $n^{it}$ and they did the $\widetilde{\chi}$ coloring for $p=2$. The last coloring is original.}

 %\begin{lemma}
%  Let $a,c\in\N$, $b,d\in\Z$. If either $a\neq c$, or $a=c$ but $a\nmid \mathrm{lcm}(b,d)$,  then there exists a completely multiplicative function $f:\N\to \mathbb{S}^1$ for which
%  \[ \liminf_{n\to\infty} |f(an+b)-f(cn+d)| > 0. \]
% \end{lemma}

\begin{lemma}
Let $a\in\N$ and $b,d\in\Z$. Suppose $v_{p}(a) > \max\{v_p(b),v_p(d)\}$ and $v_p(b) \neq v_p(d)$ for some prime $p.$ Then there exists a completely multiplicative function $f:\N\to \mathbb{S}^1$ such that
  \[ \liminf_{n\to\infty} |f(an+b)-f(an+d)| > 0. \]
\end{lemma}

 \begin{proof}
 Let $k=|v_p(b)-v_p(d)|.$  Consider the completely multiplicative function \( f : \mathbb{N} \to \mathbb{S}^1 \) defined on primes by
 \begin{align*}
 f(q)=
\begin{cases}
 e^{\pi i /k} &\text{if $q=p,$} \\
 \hfil 1 &\text{if $q \neq p.$}
\end{cases}
 \end{align*}
 Since by assumption $|v_p \left( \frac{an+b}{an+d}\right)|=k$ for any integer $n \geq 1,$
we have
\begin{align*}
 |f(an+b)-f(an+d)| = \bigg|\frac{f(an+b)}{f(an+d)} -1\bigg| = 2, 
\end{align*}
and the lemma follows.
 \end{proof}

\section{Proof of Theorem \ref{thm:diophantine}: Sufficiency} \label{section:thm2suff}

 %Given a completely multiplicative function $f:\N \to \mathbb{S}^1$ and $\eps>0$, define the set
% \[ \mathcal{A}_{a,b,c,d}(f,\varepsilon) := \{n\in \N : |f(an+b)-f(cn+d)| < \varepsilon\}  \]
In this section, we establish the ``if" direction of Theorem \ref{thm:diophantine}. Let \( a \in \mathbb{N} \) and \( b, d \in \mathbb{Z} \) with \( a \mid \mathrm{lcm}(b,d) \). Then by Lemma~\ref{lem:reduction}, there exist \( A, B \in \mathbb{N} \) such that \( B \geq 2 \) and \( A = B(B-1) \), as well as \( C, D \in \mathbb{N} \) such that if \( n = C m + D \), then
\[
\frac{a n + b}{a n + d} = \frac{A m + B}{A m + (B - 1)}
\]
for any integer \( m \geq 1 \). In particular, for any unimodular completely multiplicative function $f$, we have 
\begin{align*}
|f(an+b)-f(an+d)| = |f(Am+B)-f(Am+(B-1))|,
\end{align*}
so that
 \[ C\cdot \mathcal{A}_{A,B,A,B-1}(f_1,\ldots,f_r;\eps) + D \subseteq \mathcal{A}_{a,b,a,d}(f_1,\ldots,f_r;\eps) \]
 for any unimodular completely multiplicative functions $f_1,\ldots,f_r.$
 Therefore, it suffices to establish positive lower asymptotic density of the set $\mathcal{A}_{A,B,A,B-1}(f_1,\ldots,f_r;\eps)$.

\begin{proposition}
Let \( r, a, b \in \mathbb{N} \) with \( a = b(b - 1) \) and \( b \geq 2 \). Then for any \( \varepsilon > 0 \) and any completely multiplicative functions \( f_1, \ldots, f_r : \mathbb{N} \to \mathbb{S}^1 \), the set \( \mathcal{A}_{a,b,a,b-1}(f_1, \ldots, f_r; \varepsilon) \) has positive lower asymptotic density.

\end{proposition}
 \begin{proof}
     Fix $0<\varepsilon<1/3$, let $k\geq 1$ be the smallest integer for which $|e^{2\pi i/k}-1| < \varepsilon$. Write $f_j(n) = e^{i \arg f_j(n)}$ for $1 \leq j \leq r.$ Consider the $k^r$-coloring $c:\mathbb{N} \to \{1, \ldots,k \}^r$ given by
     \begin{align*}
 c(n) = (\ell_1, \ldots, \ell_r) \quad\text{if } \, \frac{1}{2\pi}\arg f_j(n) \Mod{1}  \in \bigg[\frac{\ell_j-1}{k}, \frac{\ell_j}{k} \bigg) \quad \forall \, 1\leq j \leq r.
     \end{align*}
     %For each interval $I_1 = \{1,\ldots, H_{k+1}\}$, $I_2 = \{H_{k+1}+1,\ldots, 2H_{k+1}\}$, $\ldots$, $I_P = \{(P-1)H_{k+1} + 1, \ldots, H_{k+1}\}$, $\ldots$, 
Write $I_t=\{ (t-1)H_{k^r+1}+1, \ldots, tH_{k^r+1} \}$ for integers $t \geq 1$, where $H_{k^{r}+1}$ is as in Proposition \ref{prop:a=b(b-1)}. Applying Proposition \ref{prop:a=b(b-1)}, the induced subgraph 
    $G_{a,b,a,b-1}[I_t]$ contains a $(k^r+1)$-clique for each integer $t \geq 1.$ Since $c$ is a $k^{r}$-coloring, at least two vertices of these cliques must share the same color by the pigeonhole principle. That is, there exist integers $x_t, y_t \in I_t$ that are adjacent as vertices such that $c(x_t)=c(y_t)$. This implies
    \begin{align} \label{eq:xtyt}
   \max_{1 \leq j \leq r}  \bigg| \frac{f_j(x_t)}{f_j(y_t)} -1 \bigg| < | e^{2\pi i/k} -1| < \varepsilon. 
    \end{align}
  Since $x_t, y_t$ are adjacent as vertices, there exists an integer $n_t \geq 1$ such that
  \begin{align} \label{eq:ntnt}
   \frac{x_t}{y_t} = \frac{an_t + b}{an_t+(b-1)}. 
  \end{align}
Let $g_t:=\operatorname{gcd}(x_t, y_t).$ Then, since
\begin{align} \label{eq:gta}
 x_t = g_t(an_t+b)  \quad \text{and} \quad y_t = g_t(an_t + (b-1)),
\end{align} 
 we have $g_t=x_t-y_t \leq H_{k^r+1}$. Therefore, by the pigeonhole principle, for each integer $J \geq 1,$  there exist $1 \leq g \leq H_{k^r+1}$ and $1 \leq t_1, \ldots, t_J \leq (J-1)H_{k^r+1}+1$ such that $g_{t_i}=g$ for $1 \leq i \leq J.$
Since \( x_{t_i}, y_{t_i} \in I_{t_i} \), it follows from \eqref{eq:gta} that the integers \( n_{t_i} \) are pairwise distinct for \( 1 \leq i \leq J \). Finally, since
 %    \[ \{n_{t_1},\ldots, n_{t_J}\} \subseteq \bigcup_{t=1}^{JH_{k^r+1}} I_t = \{1,\ldots,JH_{k^r+1}^2\} \]
 $1 \leq n_{t_1},\ldots, n_{t_J} \leq JH_{k^r+1}^2,$
it follows from (\ref{eq:xtyt}) and (\ref{eq:ntnt}) that 
\begin{align*}
|\mathcal{A}_{a,b,a,b-1}(f_1, \ldots, f_r;\varepsilon) \cap \{1,\ldots,JH_{k^r+1}^2\}| \geq J. 
\end{align*}
Therefore, we have
     \begin{align*}
       \underline{d}(\mathcal{A}_{a,b,a,b-1}(f_1, \ldots, f_r;\varepsilon)) =&  \liminf_{n\to\infty} \frac{|\mathcal{A}_{a,b,a,b-1}(f_1, \ldots, f_r;\varepsilon)\cap\{1,\ldots,n\}|}{n} \\
       \geq & \liminf_{n\to\infty} \frac{\lfloor n/ H_{k^r+1}^2\rfloor}{n} \\
         = & \frac{1}{H_{k^r+1}^2},
     \end{align*}
    and the proposition follows.
\end{proof}

%\section{Measure multiplicative recurrence} \label{sec:pretend}

%\begin{definition}[Set of measure multiplicative recurrence] \label{def:measmultrec}
%\end{definition}

%\begin{theorem}[Furstenberg's correspondence principle]
%\label{thm:furstenberg}
%\end{theorem}

\section*{Acknowledgements}
The authors are grateful to Andrew Granville for his advice and encouragement. They would like to thank Oleksiy Klurman, Cihan Sabuncu, and Stelios Sachpazis for helpful discussions. They also thank Dimitrios Charamaras, Nikos Frantzikinakis, Andreas Mountakis, and Konstantinos Tsinas for carefully reading an earlier version of the manuscript. In particular, they are thankful to Joel Moreira for pointing out that the existence of arbitrarily large cliques yields not only topological but also measurable multiplicative recurrence, thereby significantly strengthening their earlier results.

%\nocite{*}
\printbibliography

%\vspace{-\baselineskip}

\end{document}